\documentclass[12pt,reqno]{amsart}
\usepackage{amsmath, amsthm, amssymb, stmaryrd}
\usepackage{hyperref}
\usepackage{enumerate}
\usepackage{enumitem}
\usepackage{url}
\usepackage{tikz-cd}

\let\OLDthebibliography\thebibliography
\renewcommand\thebibliography[1]{
  \OLDthebibliography{#1}
  \setlength{\parskip}{0pt}
  \setlength{\itemsep}{0pt plus 0.3ex}
}

\usepackage{mathtools}

\usepackage{multirow, array}
\usepackage{placeins}

\usepackage{caption} 
\advance \topmargin by -\headheight
\advance \topmargin by -\headsep
     
\evensidemargin \oddsidemargin
\newtheorem{thm}{Theorem}[section]
\newtheorem{lemma}[thm]{Lemma}
\newtheorem{prop}[thm]{Proposition}

\newtheorem{conj}[thm]{Conjecture}

\theoremstyle{definition}
\newtheorem{defn}[thm]{Definition}
\newtheorem{example}[thm]{Example}

\theoremstyle{remark}

\numberwithin{equation}{section}

\newcommand*\wrapletters[1]{\wr@pletters#1\@nil}
\def\wr@pletters#1#2\@nil{#1\allowbreak\if&#2&\else\wr@pletters#2\@nil\fi}

\def \bbF {\mathbb F}

\def \bQ {\mathbb Q}

\def \bZ {\mathbb Z}

\def \fP {\mathfrak P}

\def \det {\mathrm{det}}

\def \Gal {{\mathrm{Gal}}}
\def \sgn {{\mathrm{sgn}}}

\begin{document}
\title[On the density of a set of primes associated to an elliptic curve]{On the density of a set of primes associated to an elliptic curve}
\author[Nuno Arala]{Nuno Arala}
\email{Nuno.Arala-Santos@warwick.ac.uk}
\thanks{}
\date{}
\begin{abstract} For a given elliptic curve $E$ defined over the rationals, we study the density of primes $p$ satisfying $\mathrm{gcd}(\#E(\bbF_p),p-1)=1$ and give a conjectural value for this density with strong heuristic evidence for most elliptic curves, in an appropriate sense.

\end{abstract}

\maketitle

\section{Introduction}
\noindent As motivation for the considerations that follow, we quote the following result, which gives a sufficient condition for an elliptic curve $E/\mathbb{Q}$ not to have integral points over any cyclotomic ring $\mathbb{Z}[\zeta_{p^n}]$, where $p$ is a fixed prime and $\zeta_{p^n}$ denotes a primitive $p^n$-th root of unity.

\begin{thm}
Let $E/\mathbb{Q}$ be an elliptic curve with $E(\mathbb{Q})=\{0\}$ and with conductor $\mathcal{N}_E$. Define
$$R_E=\{p\nmid\mathcal{N}_E\text{ prime}:\mathrm{gcd}(p(p-1),\#E(\mathbb{F}_p))=1\}\text{.}$$
Then $(E-0)(\mathbb{Z}[\zeta_{p^n}])=\emptyset$ for $p\in R_E$ and $n\geq1$.
\end{thm}

\begin{proof}
This is a particular case of Corollary 4 in \cite{samir}.
\end{proof}
Here $(E-0)(\mathbb{Z}[\zeta_{p^n}])$ denotes the set of points in $E(\mathbb{Q}(\zeta_{p^n}))$ which do \emph{not} reduce to the identity point $0$ in $E(\mathbb{Z}[\zeta_{p^n}]/\mathfrak{P})$ for any prime ideal $\mathfrak{P}$ of $\mathbb{Z}[\zeta_{p^n}]$. If $y^2+a_1xy+a_3y=x^3+a_2x^2+a_4x+a_6$ is a minimal Weierstrass equation for $E/\mathbb{Q}$, then the conclusion of the previous theorem can be phrased by stating that the equation $y^2+a_1xy+a_3y=x^3+a_2x^2+a_4x+a_6$ has no solutions in $(x,y)\in(\mathbb{Z}[\zeta_{p^n}])^2$.

\bigskip

The previous theorem makes it natural to attempt to estimate the density of primes contained in $R_E$, for a fixed $E$. Obviously the condition $p\nmid\mathcal{N}_E$ is unimportant for density considerations since it only affects a finite set of primes. Moreover the set of primes such $p$ such that $\mathrm{gcd}(p,\#E(\mathbb{F}_p))>1$ (i.e. such that $p\mid\#E(\mathbb{F}_p)$) has density $0$. This can be seen as follows. It is well-known that
$$\#E(\mathbb{F}_p)=p+1-a_p$$
where $a_p$ is the trace of the Frobenius map acting on the reduction of $E$ over $\mathbb{F}_p$. The Hasse bound (\cite{silverman}, Theorem V.1.1) states that $|a_p|<2\sqrt{p}$. On the other hand, if $p$ divides $\#E(\mathbb{F}_p)$, then we must have $a_p\equiv 1\pmod{p}$, and for $p\geq 7$ this together with $|a_p|<2\sqrt{p}$ forces $a_p=1$. Now the Sato-Tate Conjecture (\cite{lgg}, \cite{lggt}, \cite{hs-rt}) easily implies that the set of primes $p$ for which $a_p=1$ has density $0$. We reduce therefore our question to the following problem.

\bigskip

\noindent\textbf{Problem.} Given an elliptic curve $E/\mathbb{Q}$, evaluate the density of primes $p$ satisfying
$$\mathrm{gcd}(p-1,\#E(\mathbb{F}_p))=1\text{.}$$

\bigskip

In what follows, we give strong evidence (namely, Theorem \ref{maincount}, which is the main result in the paper) towards the following conjectural (partial) answer to our problem (for the definition of a Serre curve, see Proposition \ref{serreprop} and the discussion after that).

\begin{conj}
Define
$$C=\prod_{\ell\text{ prime}}\left(1-\frac{\ell}{(\ell-1)^2(\ell+1)}\right)\text{.}$$
Let $E/\bQ$ be a Serre curve, and let $D$ be the discriminant of the number field $\bQ(\sqrt{\Delta})$, where $\Delta$ is the discriminant of $E$. Then the density of primes $p$ for which $\mathrm{gcd}(\#E(\bbF_p),p-1)=1$ is
$$\begin{cases}
C&\text{ if }D\equiv0\pmod{4}\\
\left(1+\prod_{\ell\mid D}\frac{-\ell}{\ell^3-\ell^2-2\ell+1}\right)C&\text{ if }D\equiv1\pmod{4}\text{.}
\end{cases}$$
\end{conj}

\section{Notation}
\label{notation}

Our approach will make use of Galois representations associated to the torsion of $E$. We fix the corresponding notation here. For a (not necessarily positive) integer $N$, we use the standard notation $E[N]$ for the subgroup of $N$-torsion points of $E$; then $E[N]\cong(\bZ/N\bZ)^2$ (\cite{silverman}, Corollary III.6.4). The absolute Galois group $G_\bQ=\Gal(\overline{\bQ}/\bQ)$ acts on $E[N]$ via $\bZ/N\bZ$-automorphisms, yielding a group homomorphism
$$\rho_{E,N}:G_\bQ\to\mathrm{GL}_2(\bZ/N\bZ)\text{.}$$
Similarly, denoting by $\bQ(E[N])$ the smallest extension of $\bQ$ over which all points of $E[N]$ are rational, the Galois group $\Gal(\bQ(E[N])/\bQ)$ acts on $E[N]$ via $\bZ/N\bZ$-automorphisms; this yields a group homomorphism $\overline{\rho}_{E,N}:\Gal(\bQ(E[N])/\bQ)\to\mathrm{GL}_2(\bZ/N\bZ)$. These two homomorphisms fit in the following commutative diagram:
\begin{center}
\begin{tikzcd}
G_\mathbb{Q} \arrow[r, two heads] \arrow[r] \arrow[rr, "{\rho_{E,N}}"', bend right] & {\mathrm{Gal}(\mathbb{Q}(E[N])/\mathbb{Q})} \arrow[r, "{\overline{\rho}_{E,N}}", hook] & \mathrm{GL}_2(\bZ/N\bZ)
\end{tikzcd}
\end{center}
By taking the inverse limit, the maps $(\rho_{E,N})_{N>0}$ yield a homomorphism
$$\rho_E:G_\bQ\to\varprojlim\mathrm{GL}_2(\bZ/N\bZ)=\mathrm{GL}_2(\widehat{\bZ})\text{.}$$

\section{The local count}
\label{thelocalcount}
The condition $\mathrm{gcd}(p-1,\#E(\mathbb{F}_p))=1$ can be restated as follows: there is no prime $\ell$ satisfying $p\equiv 1\pmod{\ell}$ and $\ell\mid\#E(\mathbb{F}_p)$. Therefore we begin by evaluating, for a fixed prime $\ell$, the density of primes $p$ satisfying $p\equiv1\pmod{\ell}$ and $\ell\mid\#E(\mathbb{F}_p)$.

Recall the Galois representation $\overline{\rho}_{E,\ell}$, which embeds $\Gal(\bQ(E[\ell])/\bQ)$ into $\mathrm{GL}_2(\bbF_\ell)$; we will work now under the assumption that this embedding is actually an isomorphism. The main result in this section is the following.
\begin{lemma}
\label{local}
Let $E/\mathbb{Q}$ be an elliptic curve and let $\ell$ be a prime. Suppose that $\overline{\rho}_{E,\ell}$ (or, equivalently, $\rho_{E,\ell}$) is surjective. Then the density of primes $p$ for which $\ell$ is not a common factor of $p-1$ and $\#E(\mathbb{F}_p)$ is
$$1-\frac{\ell}{(\ell-1)^2(\ell+1)}\text{.}$$
\end{lemma}
In order to prove this we need two auxiliary propositions.

\begin{prop}
\label{fprop}
Let $E/\mathbb{Q}$ be an elliptic curve and let $\ell$ be a prime. Let $p\neq\ell$ be a prime of good reduction for $E$, let $\mathfrak{P}$ denote a prime above $p$ in $\mathbb{Q}(E[\ell])$ and let $\sigma_\mathfrak{P}$ denote the corresponding Frobenius element. Then $\ell$ divides $\#E(\mathbb{F}_p)$ if and only if $1$ is an eigenvalue of $\overline{\rho}_{E,\ell}(\sigma_\mathfrak{P})$.
\end{prop}
\begin{proof}
This is implicit in \cite{zywina}, \S2, but we write a full proof for completeness. Assume first that $1$ is an eigenvalue of $\overline{\rho}_{E,\ell}(\sigma_\mathfrak{P})$. Then $\sigma_\mathfrak{P}$ has a fixed point, say $Q$, in $E[\ell]$. Since $\sigma_\mathfrak{P}(Q)=Q$, reducing this modulo $\mathfrak{P}$ yields $\phi_p(Q')=Q'$, where $Q'$ is the reduction of $Q$ modulo $\mathfrak{P}$ and $\phi_p$ is the Frobenius automorphism of $\overline{\mathbb{F}_p}$ given by raising to the $p$-th power. This equality implies that $Q'\in E(\mathbb{F}_p)$. Since $Q\in E[\ell]$ and the reduction map restricted to $E[\ell]$ is injective (\cite{silverman}, Proposition IV.3.1) it follows that $Q'$ has order $\ell$. Therefore $E(\mathbb{F}_p)$ has a point of order $\ell$. Hence by Lagrange's Theorem $\ell\mid\#E(\mathbb{F}_p)$.

Conversely, assume that $\ell\mid\#E(\mathbb{F}_p)$. Then by Cauchy's Theorem $E(\mathbb{F}_p)$ contains a point $Q'$ of order $\ell$. We lift $Q'$ to a point $Q$ in $E[\ell]$. Then the fact that $\phi_p(Q')=Q'$ together with the definition of the Frobenius element $\sigma_\mathfrak{P}$ implies
$$\sigma_\mathfrak{P}(Q)\equiv Q\pmod{\mathfrak{P}}\text{.}$$
But injectivity of reduction modulo $\mathfrak{P}$ on $E[\ell]$ implies that we actually have $\sigma_\mathfrak{P}(Q)=Q$. Therefore $\sigma_\mathfrak{P}$ has a fixed point in $E[\ell]$, which means precisely that $\overline{\rho}_{E,\ell}(\sigma_{\mathfrak{P}})$ has $1$ as an eigenvalue.
\end{proof}

Similarly, we now express the condition that $\ell\mid p-1$ in terms of the image of $\sigma_\mathfrak{P}$ under the Galois representation $\overline{\rho}_{E,\ell}$.

\begin{prop}
\label{fprop2}
Let $E/\mathbb{Q}$ be an elliptic curve and let $\ell$ be a prime. Let $p\neq\ell$ be a prime, let $\mathfrak{P}$ denote a prime above $p$ in $\mathbb{Q}(E[\ell])$ and let $\sigma_\mathfrak{P}$ denote the corresponding Frobenius element. Then $\ell$ divides $p-1$ if and only if $\det(\overline{\rho}_{E,\ell}(\sigma_\mathfrak{P}))=1$ (i.e. if $\overline{\rho}_{E,\ell}(\sigma_\mathfrak{P})\in\mathrm{SL}_2(\mathbb{F}_\ell)$).
\end{prop}
\begin{proof}
We first recall that $\mathbb{Q}(\zeta_\ell)\subseteq\mathbb{Q}(E[\ell])$, where $\zeta_\ell$ denotes a primitive $\ell$-th root of unity. This is a consequence of basic properites of the Weil pairing (see \cite{silverman}, section III.8). Therefore $\mathrm{Gal}(\mathbb{Q}(E[\ell])/\mathbb{Q})$ surjects naturally onto $\mathrm{Gal}(\mathbb{Q}(\zeta_\ell)/\mathbb{Q})\cong\mathbb{F}_\ell^\times$. Under $\overline{\rho}_{E,\ell}$, $\Gal(\bQ(E[\ell])/\bQ)$ is identified with a subgroup of $\mathrm{GL}_2(\bbF_\ell)$, so we get a map from a subgroup of $\mathrm{GL}_2(\bbF_\ell)$ to $\bbF_\ell^\times$. It is a standard fact that this map is given by the determinant. In other words, we have a commutative diagram:
\begin{center}
\begin{tikzcd}
{\mathrm{Gal}(\mathbb{Q}(E[\ell])/\mathbb{Q})} \arrow[rrd, two heads] \arrow[rr, "{\overline{\rho}_{E,\ell}}", hook] &  & \mathrm{GL}_2(\mathbb{F}_\ell) \arrow[d, "\det"] \\
                                                                                                                     &  & \mathbb{F}_\ell^\times                          
\end{tikzcd}
\end{center}

This reduces our task to proving that $p\equiv1\pmod{\ell}$ if and only if $\sigma_\fP(\zeta_\ell)=\zeta_\ell$. Since $\sigma_\fP(\zeta_\ell)=\zeta_\ell^p$, this is clear.
\end{proof}
We are now in a position to prove Lemma \ref{local}.
\begin{proof}[Proof of Lemma \ref{local}]
Keeping the notation used in Lemmas \ref{fprop} and \ref{fprop2}, combining both these lemmas we see that we are interested in evaluating the density of primes $p$ for which $\overline{\rho}_{E,\ell}(\sigma_\fP)$ has determinant $1$ and has $1$ as an eigenvalue (this is the density of the set of ``bad'' primes, those for which $\ell$ is a common factor of $p-1$ and $\#E(\bbF_p)$). In other words, we want primes $p$ such that $\overline{\rho}_{E,\ell}(\sigma_\fP)$ is conjugate to a matrix of the form
$$\begin{pmatrix}1&\ast\\0&1\end{pmatrix}\text{.}$$
We count these matrices: apart from the identity matrix, all these matrices fall into the conjugacy class of $T=\begin{pmatrix}1&1\\0&1\end{pmatrix}$. In fact, for $a\in\mathbb{F}_\ell^\star$, a calculation shows that
$$\begin{pmatrix}1&0\\0&a\end{pmatrix}^{-1}\begin{pmatrix}1&1\\0&1\end{pmatrix}\begin{pmatrix}1&0\\0&a\end{pmatrix}=\begin{pmatrix}1&a\\0&1\end{pmatrix}\text{.}$$
We are then interested in computing the cardinalities of the conjucacy classes in $\mathbb{GL}_2(\mathbb{F}_\ell)$ of the identity and $T$. The former has cardinality $1$; we now compute the order of the stabilizer of $T$ under the conjugation action. A matrix $M=\begin{pmatrix}a&c\\b&d\end{pmatrix}\in\mathrm{GL}_2(\bbF_\ell)$ stabilizes $T$ if and only if
$$\begin{pmatrix}a&c\\b&d\end{pmatrix}\begin{pmatrix}1&1\\0&1\end{pmatrix}=\begin{pmatrix}1&1\\0&1\end{pmatrix}\begin{pmatrix}a&c\\b&d\end{pmatrix}\text{,}\quad\text{i.e.,}\quad\begin{pmatrix}a&a+c\\b&b+d\end{pmatrix}=\begin{pmatrix}a+b&c+d\\b&d\end{pmatrix}\text{.}$$
The previous equality is equivalent to $a=d$ and $b=0$. This gives $\ell^2$ possible matrices, but those $\ell$ of them for which $a=d=0$ are not in $\mathrm{GL}_2(\bbF_\ell)$, so we conclude that the stabilizer of $T$ has $\ell^2-\ell$ elements. By the orbit-stabilizer theorem together with the well-known fact that $\#\mathrm{GL}_2(\mathbb{F}_\ell)=(\ell^2-1)(\ell^2-\ell)$, we conclude that the conjugacy class of $T$ has
$$\frac{\#\mathrm{GL}_2(\bbF_\ell)}{\ell^2-\ell}=\ell^2-1$$
elements. Therefore there are $\ell^2$ matrices in $\mathrm{GL}_2(\mathbb{F}_\ell)$ which have determinant $1$ and have $1$ as eigenvalue.

Now we use Chebotarev's density theorem: by assumption $\overline{\rho}_{E,\ell}$ identifies $\Gal(\bQ(E[\ell])/\bQ)$ with $\mathrm{GL}_2(\bbF_\ell)$, and we want to count primes $p$ for which $\overline{\rho}_{E,\ell}(\sigma_\fP)$ falls outside one of two conjugacy classes of $\mathrm{GL}_2(\bbF_\ell)$ under this identification. By the computation above, these conjugacy classes altogether have $\ell^2$ elements, and since $\mathrm{GL}_2(\bbF_\ell)$ has $(\ell^2-1)(\ell^2-\ell)$ elements, the remaining conjugacy classes comprise $(\ell^2-1)(\ell^2-\ell)-\ell^2$ elements of $\mathrm{GL}_2(\bbF_\ell)$. By Chebotarev's density theorem it follows that the desired density is
$$\frac{(\ell^2-1)(\ell^2-\ell)-\ell^2}{(\ell^2-1)(\ell^2-\ell)}=1-\frac{\ell}{(\ell-1)^2(\ell+1)}\text{,}$$
as we wanted.
\end{proof}

\section{The global count}

Recall that, given an elliptic curve $E/\bQ$, we are interested in computing the density of rational primes $p$ satisfying $\mathrm{gcd}(p-1,\#E(\bbF_p))=1$. A prime $p$ satisfies this if and only if, for every prime $\ell$, the conditions $p\equiv1\pmod{\ell}$ and $\ell\mid\#E(\bbF_p)$ do not hold simultaneously. Suppose that the Galois representations $\overline{\rho}_{E,\ell}$ are surjective for all primes $\ell$. A naive probabilistic reasoning, based on Lemma \ref{local}, then suggests that the desired density is
\begin{equation}
\label{naiveconst}
\prod_{\ell\text{ prime}}\left(1-\frac{\ell}{(\ell-1)^2(\ell+1)}\right)\approx0.24238005\text{.}
\end{equation}
This naive reasoning, however, turns out to give the wrong answer, and the reason is that it assumes some sort of independence between the events
$$(p\equiv 1\pmod{\ell}\text{ and }\ell\mid\#E(\bbF_p))$$
as $\ell$ ranges over the primes. This does not necessarily hold, and we need to add some factors to our prediction \eqref{naiveconst} (the so-called entanglement correction factors, following the standard terminology in the literature, as laid out for example in \cite{daniels}) to account for this lack of independence.

To explain this lack of independence, let $\ell_1$ and $\ell_2$ be distinct primes. As we saw before, in Lemmas \ref{fprop} and \ref{fprop2}, whether $\ell_1$ divides $\#E(\bbF_p)$ is determined by the Frobenius over $p$ in $\Gal(\bQ(E[\ell_1])/\bQ)$, and, similarly, whether $\ell_2$ divides $\#E(\bbF_p)$ is determined by the Frobenius over $p$ in $\Gal(\bQ(E[\ell_2])/\bQ)$. Denote these two automorphisms by $\sigma_{p,\ell_1}$ and $\sigma_{p,\ell_2}$, respectively. Then $\sigma_{p,\ell_1}$ and $\sigma_{p,\ell_2}$ are the restrictions to $\bQ(E[\ell_1])$ and $\bQ(E[\ell_2])$, respectively, of an element of $\Gal(\bQ(E[\ell_1\ell_2])/\bQ)$, namely the Frobenius $\sigma_{p,\ell_1\ell_2}$ over $p$ in $\Gal(\bQ(E[\ell_1\ell_2])/\bQ)$.
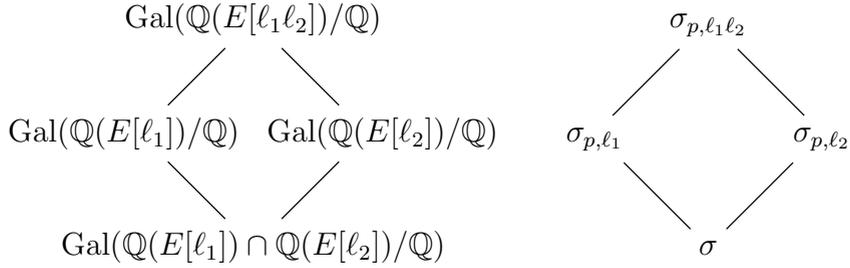
\begin{center}
    \begin{tikzcd}[column sep={1.5cm,between origins},row sep={1.5cm,between origins}]
                                                                     & {\mathrm{Gal}(\mathbb{Q}(E[\ell_1\ell_2])/\mathbb{Q})}                                        &                                                                                          &                                                             & & {\sigma_{p,\ell_1\ell_2}} \arrow[rd, no head] &                                         \\
{\mathrm{Gal}(\mathbb{Q}(E[\ell_1])/\mathbb{Q})\quad} \arrow[ru, no head] &                                                                                               & {\quad\mathrm{Gal}(\mathbb{Q}(E[\ell_2])/\mathbb{Q})} \arrow[ld, no head] \arrow[lu, no head] & & {\sigma_{p,\ell_1}} \arrow[ru, no head] \arrow[rd, no head] &                                               & {\sigma_{p,\ell_2}} \arrow[ld, no head] \\
                                                                     & {\mathrm{Gal}(\mathbb{Q}(E[\ell_1])\cap\mathbb{Q}(E[\ell_2])/\mathbb{Q})} \arrow[lu, no head] &                                                                                       &   &                                                             & \sigma                                        &                                        
\end{tikzcd}
\end{center}
But this implies that $\sigma_{p,\ell_1}$ and $\sigma_{p,\ell_2}$ restrict to the same element of $\mathrm{Gal}(\mathbb{Q}(E[\ell_1])\cap\mathbb{Q}(E[\ell_2])/\mathbb{Q})$. If the intersection field $\mathbb{Q}(E[\ell_1])\cap\mathbb{Q}(E[\ell_2])$ is strictly larger than $\mathbb{Q}$, then this gives a non-trivial dependence between $\sigma_{p,\ell_1}$ and $\sigma_{p,\ell_2}$, for an arbitrary prime $p$.

\begin{example}
\label{-3n2}
Suppose $E/\mathbb{Q}$ is an elliptic curve whose discriminant $\Delta$ is of the form $-3n^2$, for some integer $n$, and that $\overline{\rho}_{E,2}$ and $\overline{\rho}_{E,3}$ are surjective. Then, according to the naive heuristic that predicted \eqref{naiveconst}, the density of primes $p$ for which $2$ and $3$ do not divide $\#E(\bbF_p)$ should be
\begin{equation}
\label{1348}
\prod_{\ell=2,3}\left(1-\frac{\ell}{(\ell-1)^2(\ell+1)}\right)=\frac{13}{48}\text{.}
\end{equation}
However, this is not correct, because the fields $\bQ(E[2])$ and $\bQ(E[3])$ intersect non-trivially: on the one hand it is well-known that $\bQ(E[2])$ contains $\sqrt{\Delta}$ (and hence contains $\sqrt{-3}$, given our assumption on $\Delta$), and on the other hand, as remarked before, $\bQ(E[3])$ contains $\zeta_3=\frac{1+\sqrt{-3}}{2}$ (and hence contains $\sqrt{-3}$). Therefore $\bQ(E[2])$ and $\bQ(E[3])$ both contain $\bQ(\sqrt{-3})$:
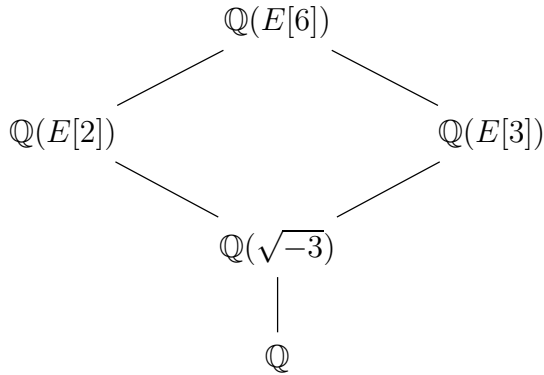
\begin{center}
\begin{tikzcd}
                                       & {\mathbb{Q}(E[6])}                                            &                                        \\
{\mathbb{Q}(E[2])} \arrow[ru, no head] &                                                               & {\mathbb{Q}(E[3])} \arrow[lu, no head] \\
                                       & \mathbb{Q}(\sqrt{-3}) \arrow[ru, no head] \arrow[lu, no head] &                                        \\
                                       & \mathbb{Q} \arrow[u, no head]                                 &                                       
\end{tikzcd}
\end{center}
This puts a restriction on the possible pairs $(\sigma_{p,2},\sigma_{p,3})$, since for any prime $p$ both of these must act on $\sqrt{-3}$ in the same way. The computations that follow imply that (if $\bQ(\sqrt{-3})$ is the full intersection $\bQ(E[2])\cap\bQ(E[3])$) the correct constant in place of \eqref{1348} is $\frac{5}{24}$.
\end{example}

More generally, if $\ell_1,\ldots,\ell_m$ and $\tilde{\ell}_1,\ldots,\tilde{\ell}_n$ are distinct primes, then if
\begin{equation}
\label{genint}
\bQ(E[\ell_1\cdots\ell_m])\cap\bQ(E[\tilde{\ell}_1\cdots\tilde{\ell}_n])
\end{equation}
is strictly larger than $\bQ$, then the $m$-tuple of Frobenius automorphisms $(\sigma_{p,\ell_1},\ldots,\sigma_{p,\ell_m})$ places a non-trivial restriction on the $n$-tuple $(\sigma_{p,\tilde{\ell}_1},\ldots,\sigma_{p,\tilde{\ell}_n})$. In general, the correction factors we need to multiply \eqref{naiveconst} with will thus depend on how intersections of the form \eqref{genint} behave as $\ell_1,\ldots,\ell_m$ and $\tilde{\ell}_1,\ldots,\tilde{\ell}_n$ range over the primes. On what follows we will give a corrected version of \eqref{naiveconst} for so-called Serre curves, the definition of which we give below. 
Recall the homomorphism
$$\rho_E:G_\bQ\to\varprojlim\mathrm{GL}_2(\bZ/N\bZ)=\mathrm{GL}_2(\widehat{\bZ})$$
(see Section \ref{notation}). This homomorphism is \emph{never} surjective, as a theorem of Serre (see \cite{serre}) shows; indeed we show below, following Serre, that there is an index $2$ subgroup of $\mathrm{GL}_2(\widehat{\bZ})$ in which $\rho_E(G_\bQ)$ is contained. The failure of $\rho_E$ to surject is in some cases explained by non-trivial intersections of the form \eqref{genint}, as we see in the proof below. We introduce a convenient piece of notation first.

\begin{defn}
Given $\gamma\in\mathrm{GL}_2(\bbF_2)$, let it act on $\bbF_2^2$ in the natural way; it then permutes the three nonzero elements of $\bbF_2^2$. We define $\sgn(\gamma)$ to be the sign of this permutation.
\end{defn}

\begin{prop}
\label{serreprop}
Let $E/\bQ$ be an elliptic curve, let $\Delta$ be its discriminant, and let $D$ be the discriminant of the $\bQ(\sqrt{\Delta})$. Then,
\begin{enumerate}[label=(\roman*)]
\item If $D\equiv 0\pmod{4}$, then there exists a primitive character $\chi$ modulo $D$ such that
$$\rho_{E,D}(G_\bQ)\subseteq\{M\in\mathrm{GL}_2(\bZ/D\bZ):\chi(\det(M))=\sgn(M\bmod{2})\}\text{,}$$
where the right hand side has index $2$ in $\mathrm{GL}_2(\bZ/D\bZ)$.
\item If $D\equiv 1\pmod{4}$, then there exists a primitive character $\chi$ modulo $D$ such that
$$\rho_{E,2D}(G_\bQ)\subseteq\{M\in\mathrm{GL}_2(\bZ/2D\bZ):\chi(\det(M))=\sgn(M\bmod{2})\}\text{,}$$
where the right hand side has index $2$ in $\mathrm{GL}_2(\bZ/2D\bZ)$.
\end{enumerate}
Therefore, in either case, there is a primitive character $\chi$ modulo $D$ such that
\begin{equation}
\label{serrecond}
\rho_E(G_\bQ)\subseteq G_E
\end{equation}
where
$$G_E=\{M\in\mathrm{GL}_2(\widehat{\bZ}):\chi(\det(M\pmod{D}))=\sgn(M\bmod{2})\}$$
has index $2$ in $\mathrm{GL}_2(\widehat{\bZ})$.
\end{prop}
\begin{proof}
We prove only the case $D\equiv 1\pmod{4}$, because it is the one that matters the most to us and the proof in the case $D\equiv0\pmod{4}$ is completely analogous. We consider the following diagram of field extensions:
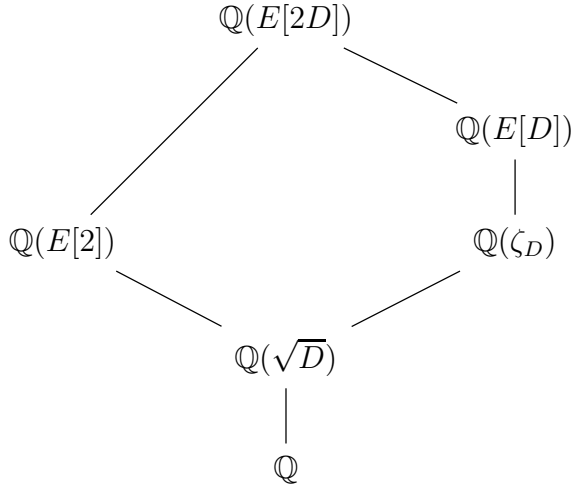
\begin{center}
\begin{tikzcd}
                                        & {\mathbb{Q}(E[2D])}                                          &                                          \\
                                        &                                                              & {\mathbb{Q}(E[D])} \arrow[lu, no head]   \\
{\mathbb{Q}(E[2])} \arrow[ruu, no head] &                                                              & \mathbb{Q}(\zeta_{D}) \arrow[u, no head] \\
                                        & \mathbb{Q}(\sqrt{D}) \arrow[ru, no head] \arrow[lu, no head] &                                          \\
                                        & \mathbb{Q} \arrow[u, no head]                                &                                         
\end{tikzcd}
\end{center}
Here we used the standard inclusion $\bQ(\sqrt{D})\subseteq\bQ(\zeta_D)$, that holds for any quadratic field discriminant $D$. Take now any $\sigma\in G_\bQ$, and let $M=\rho_{E,2D}(\sigma)$. We now describe the action of $\sigma$ on $\sqrt{D}$ in terms of $M$:
\begin{itemize}
\item On the one hand, if $y^2=x^3+Ax+B$ is a reduced Weierstrass equation for $E$, then the $2$-torsion points of $E$ are $0$, $(e_1,0)$, $(e_2,0)$ and $(e_3,0)$ where $e_1,e_2,e_3$ are the roots of $x^3+Ax+B$. Since $\sqrt{D}$ is $(e_1-e_2)(e_1-e_3)(e_2-e_3)$ up to a rational factor, $\sigma(\sqrt{D})$ equals $\sqrt{D}$ or $-\sqrt{D}$ according to whether $\sigma$ permutes $e_1,e_2,e_3$ via an even permutation or an odd permutation. The sign of this permutation is what we defined as the sign of $M\bmod{2}$:
$$\sigma(\sqrt{D})=\begin{cases}
\sqrt{D}&\text{ if }\sgn(M\bmod{2})=1\\
-\sqrt{D}&\text{ if }\sgn(M\bmod{2})=-1\text{.}
\end{cases}$$
\item On the other hand, as we remarked in the proof of Proposition \ref{fprop2}, if we identify $\Gal(\bQ(E[D])/\bQ)$ with a subgroup of $\mathrm{GL}_2(\mathbb{Z}/D\mathbb{Z})$ in the natural way then $\sigma$ acts on $\bQ(\zeta_D)$ via the automorphism corresponding to $\det(M)$ (we remarked this specifically for prime $D$, but the proof carries through in general). We now identify $\Gal(\bQ(\sqrt{D})/\bQ)$ with $\{\pm1\}$ and consider the map $\chi$ fitting in the following commutative square:
\begin{center}
\begin{tikzcd}
\mathrm{Gal}(\mathbb{Q}(\zeta_D)/\mathbb{Q}) \arrow[r, two heads] \arrow[d, "\cong"'] & \mathrm{Gal}(\mathbb{Q}(\sqrt{D})/\mathbb{Q}) \arrow[d, "\cong"] \\
(\mathbb{Z}/D\mathbb{Z})^\times \arrow[r, "\chi"', two heads]                         & \{\pm1\}                                                        
\end{tikzcd}
\end{center}
We can then regard $\chi$ as a Dirichlet character modulo $D$, and its primitivity is an easy consequence of the fact that $|D|$ is actually the \emph{smallest} positive integer $m$ such that $\bQ(\sqrt{D})\subseteq\bQ(\zeta_m)$. Now $\sigma$ fixes or moves $\sqrt{D}$ according to whether $\chi(\det(M))$ is $1$ or $-1$ respectively, so that
$$\sigma(\sqrt{D})=\begin{cases}
\sqrt{D}&\text{ if }\chi(\det(M))=1\\
-\sqrt{D}&\text{ if }\chi(\det(M))=-1\text{.}
\end{cases}$$
\end{itemize}
By comparing the two points above it follows immediately that
$$\sgn(M\bmod{2})=\chi(\det(M))\text{,}$$
as desired. It is clear that exactly half of the elements of $\mathrm{GL}_2(\bZ/D\bZ)$ have this property.
\end{proof}

We say that an elliptic curve $E/\bQ$ is a \emph{Serre curve} if equality holds in \eqref{serrecond}. As it turns out, this condition holds quite often; in \cite{jones}, N. Jones proves that almost all elliptic curves over $\bQ$ are Serre curves (for a precise statement see Theorem 4 in \cite{jones}). This makes it reasonable to focus our attention henceforth on Serre curves, which are, according to Proposition \ref{serreprop}, those for which $\rho_E(G_\bQ)$ is as large as possible. Note that for a Serre curve the maps $\overline{\rho}_{E,\ell}$, for prime $\ell$, are automatically surjective.

\begin{thm}
\label{maincount}
Let $E/\bQ$ be a Serre curve, and keep the notation from Proposition \ref{serreprop}. Let $S$ be a finite set of primes, and define
$$\delta_S=\prod_{\ell\in S}\left(1-\frac{\ell}{(\ell-1)^2(\ell+1)}\right)\text{.}$$
Then,
\begin{enumerate}[label=(\roman*)]
\item If $D\equiv0\pmod{4}$, the density of primes $p$ for which the conditions $\ell\mid\#E(\bbF_p)$ and $p\equiv 1\pmod{\ell}$ do not hold simultaneously for any $\ell\in S$ is $\delta_S$.
\item If $D\equiv1\pmod{4}$, the density of primes $p$ for which the conditions $\ell\mid\#E(\bbF_p)$ and $p\equiv1\pmod{\ell}$ do not hold simultaneously for any $\ell\in S$ is
$$\begin{cases}
\delta_S&\text{ if }2D\nmid\prod_{\ell\in S}\ell\\
\left(1+\prod_{\ell\mid D}\frac{-\ell}{\ell^3-\ell^2-2\ell+1}\right)\delta_S&\text{ if }2D\mid\prod_{\ell\in S}\ell\text{.}
\end{cases}$$
\end{enumerate}
\end{thm}
\begin{proof}
In what follows we denote by $T$ the set of prime factors of $2D$ and we denote by $\Pi$ the product $\prod_{\ell\in S}\ell$. If the natural composite injection
\begin{center}
\begin{tikzcd}
{\mathrm{Gal}(\mathbb{Q}(E[\Pi])/\mathbb{Q})} \arrow[r, hook] & {\prod_{\ell\in S}\mathrm{Gal}(\mathbb{Q}(E[\ell])/\mathbb{Q})} \arrow[r, "\cong"] & \prod_{\ell\in S}\mathrm{GL}_2(\mathbb{F}_\ell)
\end{tikzcd}
\end{center}
is an isomorphism, then Lemma \ref{local} extends multiplicatively: by Propositions \ref{fprop} and \ref{fprop2} and Chebotarev's density theorem, we are reduced to counting tuples consisting of one element of $\mathrm{GL}_2(\bbF_\ell)$ for each $\ell\in S$, such that the given element of $\mathrm{GL}_2(\bbF_\ell)$ is not conjugate to a matrix of the form
$$\begin{pmatrix}1&\ast\\0&1\end{pmatrix}\text{.}$$
Thus the required density is the product of local densities as computed in the proof of Lemma \ref{local}, i.e. $\delta_S$. We will then show that the map above is an isomorphism when $D\equiv0\pmod{4}$ and also when $D\equiv1\pmod{4}$ and not all prime factors of $2D$ are in $S$, establishing part (i) and the first case of part (ii).

Suppose $D\equiv0\pmod{4}$; let $2^k$ be the highest power of $2$ that divides $D$. Then $k\geq 2$ (in fact $k\in\{2,3\}$). Then the Dirichlet character $\chi$ from Proposition \ref{serreprop} factors as
$$\chi_{2^k}\prod_{\substack{\ell\mid D\\\ell\neq2}}\chi_\ell$$
where $\chi_{2^k}$ is a real primitive Dirichlet character modulo $2^k$ and $\chi_\ell$ is a real primitive Dirichlet character modulo $\ell$ for each odd prime $\ell$ dividing $D$. Take a tuple
$$(M_\ell)_{\ell\in S}\in\prod_{\ell\in S}\mathrm{GL}_2(\bbF_\ell)\text{;}$$
we want to show that this tuple lies in the image of $\rho_{E,\Pi}$, where we identified $\mathrm{GL}_2(\Pi)$ with $\prod_{\ell\in S}\mathrm{GL}_2(\bbF_\ell)$. Extend this tuple to a tuple $M=(M_\ell)_{\ell\in S\cup T}\in\prod_{\ell\in S\cup T}\mathrm{GL}_2(\bbF_\ell)$ in an arbitrary way, and denote by $\Pi'$ the product of the primes in $S\cup T$. In light of the commutative diagram
\begin{center}
\begin{tikzcd}
{\mathrm{Gal}(\mathbb{Q}(E[\Pi'])/\mathbb{Q})} \arrow[r, "{\rho_{E,\Pi'}}"] \arrow[d, two heads] & \mathrm{GL}_2(\mathbb{Z}/\Pi'\mathbb{Z}) \arrow[r, "\cong"] & \prod_{\ell\in S\cup T}\mathrm{GL}_2(\mathbb{F}_\ell) \arrow[d, two heads] \\
{\mathrm{Gal}(\mathbb{Q}(E[\Pi])/\mathbb{Q})} \arrow[r, "{\rho_{E,\Pi}}"']                       & \mathrm{GL}_2(\mathbb{Z}/\Pi\mathbb{Z}) \arrow[r, "\cong"'] & \prod_{\ell\in S}\mathrm{GL}_2(\mathbb{F}_\ell)                           
\end{tikzcd}
\end{center}
where the right vertical map is the obvious restriction, it suffices to show that this tuple lies in the image of $\rho_{E,\Pi'}$. We now use the assumption that $E$ is a Serre curve, implying that $\rho_E(G_\bQ)=G_E$ in the notation of Proposition \ref{serreprop}, and therefore that it suffices to show that there exists a lift $M_{2^k}$ of $M_2$ in $\mathrm{GL}_2(\bZ/2^k\bZ)$ such that
\begin{equation}
\label{m4}\chi_{2^k}(\det(M_{2^k}))\prod_{\ell\in T\setminus\{2\}}\chi_\ell(\det(M_\ell))=\sgn(M_2)\text{.}
\end{equation}
Take any lift
$$M_{2^k}=\begin{pmatrix}a&b\\c&d\end{pmatrix}\in\mathrm{GL}_2(\bZ/2^k\bZ)\text{.}$$
If \eqref{m4} holds we are done. Otherwise take $\omega\in(\mathbb{Z}/2^k\mathbb{Z})^\times$ such that $\chi_{2^k}(\omega)=-1$, and take
$$M'_{2^k}=\begin{pmatrix}a&\omega b\\c&\omega d\end{pmatrix}\text{.}$$
Then $M'_{2^k}$ is another lift of $M_2$ but $\det(M'_{2^k})=\omega\det(M_{2^k})$, so $\chi_{2^k}(\det(M'_{2^k}))=-\chi_{2^k}(\det(M_{2^k}))$, and \eqref{m4} holds with $M_{2^k}$ replaced by $M'_{2^k}$.

Suppose now $D\equiv1\pmod{4}$ and not all prime factors of $2D$ are in $S$. This time the character $\chi$ factors as
$$\chi=\prod_{\ell\mid D}\chi_\ell$$
where each $\chi_\ell$ is a real primitive character modulo $\ell$ for each $\ell\mid D$. Once more, let
$$(M_\ell)_{\ell\in S}\in\prod_{\ell\in S}\mathrm{GL}_2(\bbF_\ell)\text{;}$$
be a tuple we want to show is in the image of $\rho_{E,\Pi}$, where again we identified $\mathrm{GL}_2(\Pi)$ with $\prod_{\ell\in S}\mathrm{GL}_2(\bbF_\ell)$ via the Chinese Remainder Theorem. Let now $q$ be a prime factor of $2D$ which does not lie in $S$, which exists by hypothesis. It suffices to show that we can extend our tuple to a tuple $(M_\ell)_{\ell\in S\cup T}$ with
\begin{equation}
\label{wt}
\prod_{\ell\in T\setminus\{2\}}\chi_\ell(\det(M_\ell))=\sgn(M_2)\text{,}
\end{equation}
for then this tuple will lie in the image of $\rho_{E,\Pi'}$, where $\Pi'=\prod_{\ell\in S\cup T}\ell$, and so will the original one. But this is clear: since by hypothesis there is a prime $q\in T\setminus S$, we can choose $M_q$ freely so that \eqref{wt} holds, once all the other $M_\ell$ have been prescribed.

We now come to the main part of the proof, namely the case where $D\equiv1\pmod{4}$ and $S$ contains all prime factors of $2D$. In this case we can identify $\rho_{E,\Pi}$ with
$${\prod_{\ell\in S}}^\ast\mathrm{GL}_2(\bbF_\ell)=\left\{(M_\ell)_{\ell\in S}:\prod_{\ell\mid D}\chi_\ell(\det(M_\ell))=\sgn(M_2)\right\}$$
and by Propositions \ref{fprop} and \ref{fprop2} and Chebotarev's density theorem, the desired density is
$$\frac{\#\left\{(M_\ell)_{\ell\in S}\in\prod_{\ell\in S}^\ast\mathrm{GL}_2(\bbF_\ell):M_\ell\nsim\begin{pmatrix}1&\ast\\0&1\end{pmatrix}\right\}}{\#\prod_{\ell\in S}^\ast\mathrm{GL}_2(\bbF_\ell)}\text{.}$$
We wish to compare this against
$$\delta_S=\frac{\#\left\{(M_\ell)_{\ell\in S}\in\prod_{\ell\in S}\mathrm{GL}_2(\bbF_\ell):M_\ell\nsim\begin{pmatrix}1&\ast\\0&1\end{pmatrix}\right\}}{\#\prod_{\ell\in S}\mathrm{GL}_2(\bbF_\ell)}\text{.}$$
Since $\prod_{\ell\in S}^\ast\mathrm{GL}_2(\bbF_\ell)$ has index $2$ in $\prod_{\ell\in S}\mathrm{GL}_2(\bbF_\ell)$, the ratio between the desired density and $\delta_S$ is
$$\frac{2\#\left\{(M_\ell)_{\ell\in S}\in\prod_{\ell\in S}^\ast\mathrm{GL}_2(\bbF_\ell):M_\ell\nsim\begin{pmatrix}1&\ast\\0&1\end{pmatrix}\right\}}{\#\left\{(M_\ell)_{\ell\in S}\in\prod_{\ell\in S}\mathrm{GL}_2(\bbF_\ell):M_\ell\nsim\begin{pmatrix}1&\ast\\0&1\end{pmatrix}\right\}}$$
and it is clear that each prime $\ell\nmid 2D$ contributes with a factor $\#\mathrm{GL}_2(\bbF_\ell)$ to both the numerator and the denominator above, so these primes can be ignored, and we are left to compute
\begin{equation}
\label{corr}\frac{2\#\left\{(M_\ell)_{\ell\mid 2D}\in\prod_{\ell\in S}^\ast\mathrm{GL}_2(\bbF_\ell):M_\ell\nsim\begin{pmatrix}1&\ast\\0&1\end{pmatrix}\right\}}{\#\left\{(M_\ell)_{\ell\mid 2D}\in\prod_{\ell\in S}\mathrm{GL}_2(\bbF_\ell):M_\ell\nsim\begin{pmatrix}1&\ast\\0&1\end{pmatrix}\right\}}\text{.}
\end{equation}
According to the counting employed in the proof of Lemma \ref{local}, the denominator of \eqref{corr} is $\prod_{\ell\mid 2D}(\#\mathrm{GL}_2(\bbF_\ell)-\ell^2)$. We now observe that we can rewrite the numerator as
$$\sum\left(1+\sgn(M_2)\prod_{\ell\mid D}\chi_\ell(\det(M_\ell))\right)$$
where the sum ranges over tuples $(M_\ell)_{\ell\mid2D}$ such that each $M_\ell$ is not conjugate to a matrix of the form $\begin{pmatrix}1&\ast\\0&1\end{pmatrix}$ for each $\ell$. Indeed, the summand on display equals $2$ if the corresponding tuple is in $\prod_{\ell\in S}^\ast\mathrm{GL}_2(\bbF_\ell)$ and $0$ otherwise. This equals
\begin{align*}
&\sum1+\sum\sgn(M_2)\prod_{\ell\mid D}\chi_\ell(\det(M_\ell))\\
&=\prod_{\ell\mid 2D}(\#\mathrm{GL}_2(\bbF_\ell)-\ell^2)+\left(\sum\sgn(M_2)\right)\prod_{\ell\mid D}\left(\sum\chi_\ell(\det(M_\ell))\right)\text{,}
\end{align*}
where the condition on the summands is the same as the one specified before.

One computes directly that $\sum\sgn(M_2)=2$. Moreover, if $\ell\mid D$ then \emph{all} matrices in $\mathrm{GL}_2(\bbF_\ell)$ that are conjugate to a matrix of the form $\begin{pmatrix}1&\ast\\0&1\end{pmatrix}$ have determinant $1$ and as such they lie on the kernel of $\chi_\ell\circ\det$. Therefore, among the matrices considered in our summation, there are $\frac{1}{2}\#\mathrm{GL}_2(\bbF_\ell)-\ell^2$ matrices on which $\chi_\ell\circ\det$ evaluates to $1$ and $\frac{1}{2}\#\mathrm{GL}_2(\bbF_\ell)$ on which it evaluates to $-1$. Hence,
$$\sum\chi_\ell(\det(M_\ell))=(\frac{1}{2}\#\mathrm{GL}_2(\bbF_\ell)-\ell^2)-\frac{1}{2}\#\mathrm{GL}_2(\bbF_\ell)=-\ell^2\text{.}$$
We conclude that \eqref{corr} equals
\begin{align*}
\frac{\prod_{\ell\mid 2D}(\#\mathrm{GL}_2(\bbF_\ell)-\ell^2)+2\prod_{\ell\mid D}(-\ell^2)}{\prod_{\ell\mid 2D}(\#\mathrm{GL}_2(\bbF_\ell)-\ell^2)}\\
&=1+\frac{2\prod_{\ell\mid D}(-\ell^2)}{\prod_{\ell\mid 2D}(\#\mathrm{GL}_2(\bbF_\ell)-\ell^2)}\text{.}
\end{align*}
Using that $\#\mathrm{GL}_2(\bbF_\ell)=(\ell^2-1)(\ell^2-\ell)$, this simplifies to
$$1+\frac{\prod_{\ell\mid D}(-\ell^2)}{\prod_{\ell\mid D}(\ell^4-\ell^3-2\ell^2+\ell)}=1+\prod_{\ell\mid D}\frac{-\ell}{\ell^3-\ell^2-2\ell+1}\text{,}$$
as desired.
\end{proof}

\section{Conclusion}
The density we would ultimately like to compute requires a version of Theorem \ref{maincount} where $S$ is allowed to be the set of all primes. We record, however, the natural conjecture that follows from Theorem \ref{maincount} and that is the expected answer to our main problem (for Serre curves).

\begin{conj}
\label{conj}
Define
$$C=\prod_{\ell\text{ prime}}\left(1-\frac{\ell}{(\ell-1)^2(\ell+1)}\right)\text{.}$$
Let $E/\bQ$ be a Serre curve, and keep the notation from Proposition \ref{serreprop}. Then the density of primes $p$ for which $\mathrm{gcd}(\#E(\bbF_p),p-1)=1$ is
$$\begin{cases}
C&\text{ if }D\equiv0\pmod{4}\\
\left(1+\prod_{\ell\mid D}\frac{-\ell}{\ell^3-\ell^2-2\ell+1}\right)C&\text{ if }D\equiv1\pmod{4}\text{.}
\end{cases}$$
\end{conj}

It is clear from Theorem \ref{maincount} that the value conjectured above is an upper bound for the desired density, provided that it exists at all. It would be tempting to try to establish Conjecture \ref{conj} by some sort of limiting procedure, using Lemma \ref{maincount} as an input. However, this does not seem feasible, and establishing Conjecture \ref{conj} likely requires bringing more sophisticated methods into play.

\section{Acknowledgements}
The author was funded through the Engineering and Physical Sciences Research Council Doctoral Training Partnership at the University of Warwick. The author would also like to thank Professor Samir Siksek for suggesting this problem, as well as for countless support during the elaboration of this manuscript.

\end{document}